\documentclass[12pt]{amsart}
\usepackage{amsthm,amssymb,amsmath,enumerate,graphicx}
\usepackage{amscd}
\usepackage{setspace}
\theoremstyle{plain}
\newtheorem {theorem}{Theorem}[section]

\theoremstyle{remark}
\newtheorem {definition}[theorem]{Definition}

\newtheorem {example}{Example}

\newtheorem {notation}[theorem]{Notation}
 \newtheorem{assertion}{Assertion}[theorem]
\begin{document}

\makeatother
\author{Ron Aharoni}
\address{Department of Mathematics\\ Technion, Haifa, Israel}
\thanks{The research of the first author was  supported by  an ISF grant, BSF grant no. 2006099 and by the Discount Bank Chair at the Technion.}
\email[Ron Aharoni]{raharoni@gmail.com}

\author{Irina Gorelik}
\address{Department of Mathematics\\ Technion, Haifa, Israel}
\thanks{The research of the first author was  supported by  an ISF grant.}
\email[Irina Gorelik]{irenag89@gmail.com}

\vspace*{3cm}
\thispagestyle{empty}

\begin{abstract}
The {\em independent domination number} $\gamma^i(G)$ of a graph $G$
is the maximum, over all independent sets $I$, of the minimal number
of vertices needed to dominate $I$. It is known \cite{abz} that in
chordal graphs $\gamma^i$ is equal to $\gamma$, the ordinary
domination number. The weighted version of this result is not true,
but we show that it does hold for interval graphs, and for the intersection (that is, line) graphs of subtrees of a given tree, where each subtree is a single edge.

\end{abstract}

\title{Weighted domination of independent sets}
\maketitle

\section{ Introduction}
 The (open) neighborhood of a vertex $v$ in a graph $G$, denoted  by $\tilde{N}(v)=\tilde{N}_G(v)$,   is the set of all vertices connected to $v$. Given a set $D$ of vertices we write  $\tilde{N}(D)$ for $\bigcup_{v \in D}\tilde{N}(v)$. Let  $N(D)=N_G(D) = \tilde{N}(D) \cup D$.
 A set $D$ of vertices in a graph $G$ is said to {\em dominate} a set $S$ of vertices if $S\subseteq N(D)$. A set dominating $V$ is simply called
   {\em dominating}.
 The minimal size of a
dominating set is denoted by $\gamma(G)$.  A set $I$ of vertices is called {\em independent} if it does not contain any edge of $G$. The maximum, over all
independent sets $I$, of the minimal size of a set dominating $I$,
is denoted by $\gamma^i(G)$. One reason for the interest in $\gamma^i$  is that it is a lower bound  on the topological connectivity of the
independence complex of a graph (see \cite{ah}).

\begin{notation}
Given a real valued function $f$ on a set $S$,  and a set $A\subseteq
S$,  let $f[A]=\sum_{a\in A}f(a)$. We write $|f|=f[S]$ and
call $|f|$  the {\em size} of $f$.
\end{notation}
Domination parameters have weighted versions.
\begin{definition}\label{def:dom}
Let $G=(V,E)$  be a graph, and let $w:V\to\mathbb{N}$  be a weight
function on $V$. A function $f:V\to\mathbb{R} $ {\em
$w$-dominates} a set $U$ of vertices if $f[N(u)]\geq w(u)$ for every $u\in U$.
We say that $f$ is {\em $w$-dominating} if it $w$-dominates $V$.

\end{definition}

\begin{definition}
 The {\em weighted domination number}
$\gamma_w(G)$  is the  minimal size of an integral  $w$-dominating
function. The fractional weighted domination number
$\gamma^*_w(G)$  is the  minimal size of a real valued  $w$-dominating
function.

The {\em independent domination number} $\gamma^i_w(G)$ is
the maximum over all independent sets $I$ of the minimal  size of an integral  function
 $w$-dominating $I$.

\end{definition}

These definitions coincide with the ordinary ones for $w\equiv 1$.

A graph is called {\em chordal} if it contains no induced cycle of length larger than $3$. A well known
characterization  of chordal graphs was proved in \cite{gavril}.

\begin{theorem}\label{gavril}
A graph is chordal if and only if it
is the line
graph of a family of subtrees of a tree.
\end{theorem}In
In \cite{abz} the following was proved:

\begin{theorem} If $G$ is chordal then  $\gamma^i(G) =\gamma(G)$.
\end{theorem}

This theorem does not extend to the weighted case, namely there are chordal graphs in which $\gamma_w^i<\gamma_w$.

\begin{example}\label{example}  Let T=(V,E) be
a three rays star, with rays of length 3, forked at their ends.  That is,  $V(T)=\{v\} \cup \bigcup_{i \le 3,j\le 4}\{a^j_i\}$ and \\ $E=\{(v,a^1_i) \mid i \le 3\} \cup \{(a^j_i,a^{j+1}_i) \mid i \le 3, ~j<3\}\cup\{(a^2_i,a^4_i)\mid i\le 3\}$.
Let $G$ be the intersection graph of four subtrees of $T$, that are given below with their weight function $w$:\\
$w(\{a^3_i,a^2_i,a^4_i\})=1$ for $i\le 3$,\\
 $w(\{a^1_i,a^2_i,a^j_i\})=2$ for $i\le 3,\;j=3,4$,\\
 $w(\{v,a^1_i,a^2_i\})=3 $  for $i\le 3$ and\\
$w(\{a^1_i,v,a^1_j\})=4$ for $1\le i<j\le 3$.

Here $\gamma_w=5$, while $\gamma^i_w=4$.
\end{example}
In this paper we show that $\gamma^i_w=\gamma_w$ in three subclasses of the class of chordal graphs:

(i) interval graphs,

(ii) the line graphs of a family of subtrees of a given tree, each consisting of a single edge, , and

(iii) split graphs.

\section{dispersed sets}

\begin{definition}
A set of vertices in a graph $G$ is said to be {\em dispersed} if every
two elements in it are at distance of  $3$ or more apart. Given an
integral weight function $w$ on $V(G)$, the maximal total weight of
a dispersed set is denoted by $\rho_w(G)$.
\end{definition}

 The fractional relaxation $\rho^*_w(G)$
 is  the solution of the following linear program:
 \\
$$ (P)~~~~~ \max \sum_{v \in V}w(v)g(v),~~  g:V \to \mathbb{R}^+ ~~\text{satisfying} ~~  g[N(v)] \le 1~~ \text{for~~ all}~~ V\in V. $$
 \\

 The dual (D) of this linear program is the program yielding
$\gamma_w^*$. Hence, by LP Duality,
$\gamma_w^*=\rho_w^*$. Clearly,    $$\rho_w \le \gamma^i_w \le \gamma_w$$
and hence if  $\rho_w=\gamma_w$ then our desired equality $\gamma^i_w = \gamma_w$ is valid. While sufficient for the validity of $\gamma^i_w = \gamma_w$, it is not a necessary condition, as the following example shows.

\begin{example}\label{ex-split}
let $G=(V,E)$ where $V=A\cup B$ such that $A=\{a_1,a_2,a_3\}$ is a clique and $\{b_1,b_2,b_3\}$ is independent and $\{b_i,a_i\},\{b_i,a_i+1\}\in E$ for $i=1,2,3$ where the calculation is modulo $3$. Let $w$ be the following weight function as follows: $w(a_i)=5$ and $w(b_i)=4$ for $i=1,2,3$. It is easy to show that $\gamma^i_w = \gamma_w$ (for example, by using that fact that $G$ is split, see the last section of the paper), but since in $G$ every dispersed set has only one vertex, $\rho_w=5$, while $\gamma_w^i=\gamma_w=6$.
\end{example}

In the next two sections we  prove the equality  $\gamma_w=\rho_w$ in two families of graphs.


\section{Interval graphs}


A graph is an {\em interval graph} if it is
the line graph of an interval hypergraph, namely its vertices are  intervals, and two vertices
are connected if the  intervals intersect. Since we
only deal with finite hypergraphs, we can assume that the underlying set
is the discrete line, rather than the real line.

\begin{theorem}
In interval graphs  $\gamma_w=\rho_w$ for   any  weights system $w$.
\end{theorem}

It clearly suffices to prove the theorem for integral $w$.
A particularly simple case is  $w\equiv 1$. In this case,
we can assume that no interval in  the hypergraph is contained in another, since removing an interval containing another does not change $\rho_w$ and $\gamma_w$ (which in this case are plainly $\rho$ and $\gamma$). When there is no containment between intervals, the order on the left endpoints of the intervals agrees with the order on the right endpoints, so we can order the intervals linearly. As noted in \cite{roberts}, listing the intervals in this order, the $(0,1)$-matrix of the linear program (P) then has the consecutive $1$s property,  and hence is totally unimodular. Thus the solutions of (P) and its dual (D) (which, as recalled, is the program expressing $\gamma^*_w$) are integral, proving $\rho_w=\gamma_w$ .

In the weighted case
this strategy does not work,
since it may be profitable in (P) to take an interval containing another, if the containing interval has larger weight.
In this case the $(0,1)$-matrix of the linear program (P) is not necessarily totally unimodular. For example, let the hypergraph consist of four intervals, three of them disjoint and the fourth contains all these three. The determinant of the matrix of (P) is then $2$.

Not having the total unimodularity tool at hand, we prove the equality $\rho_w=\gamma_w$ directly.
Since the inequality $\rho_w\leq\gamma_w$ is always true, it suffices to show that $\gamma_w\leq \rho_w$.

\begin{proof}
Let $H$ be a hypergraph of intervals, and let $G=(V,E)=L(H)$. Let $w:V\to\mathbb{N}$ be a weight function.

Enumerate the vertices of $G$, namely the intervals in $H$, as  $v_1=[x_1,y_1],v_2=[x_2,y_2],\dots ,v_n=[x_n,y_n]$, where  $y_1\leq y_2\leq\dots\leq  y_n$.
We use this enumeration to define  a
$w$-dominating function $f:V\to\mathbb{N}$.
Write $w_0=w$.
Let $v_{j_1}=v_1$ and let $u_1$ be the interval extending furthest to the right among all intervals  meeting $v_{j_1}$. Define $f_1=w_0(v_{j_1})\chi_{u_1}$, where, for a vertex $v$,  $\chi_v$ is the characteristic vector  of the set $\{v\}$. Let $w_1(v)=[ w_0(v)-f_1(u_1)]^+$ for all $v \in N(u_1)$, $w_1(v)=w_0(v)$ for all other vertices $v$. If $w_1\equiv 0$ then let $f=f_1$ - it  is then  $w$-dominating. Otherwise, let $v_{j_2}$ be the interval with positive $w_1$-value, having minimal right endpoint. Then $f_1$ $w$-dominates all the intervals $v_i$ for $i<j_2$. Let $u_2$ be an interval meeting $v_{j_2}$ and extending furthest to the right. Let $f_2=w_1(v_{j_2})\chi_{u_2}$, and define $w_2(v)=[ w_1(v)-f_2(u_2)]^+$ for all $v \in N(u_2)$, $w_2(v)=w_1(v)$ for all other vertices $v$.

Assume  the intervals $v_{j_1},\dots,v_{j_k}$, the functions $f_1,\dots,f_k$  and the weight function $w_{k}$ have been defined.  If $w_k\equiv 0$ then we  end the definition procedure, and let $f=\sum_{i=1}^k f_i$. Clearly, $f$ is $w$-dominating. Otherwise  let $v_{j_{k+1}}$ be the interval with positive $w_k$-value, having minimal right endpoint. In this case $\sum_{i=1}^k f_k$ $w$-dominates all the intervals $v_i$ for $i<j_{k+1}$. Let $u_{k+1}$ be an interval meeting $v_{j_{k+1}}$ and extending furthest to the right. Let $f_{k+1}=w_k(v_{j_{k+1}})\chi_{u_k}$ and define $w_{k+1}(v)=[w_{k}(v)-f_{k+1}(v)]^+$ for every $v\in N(u_k)$ and                                      $w_{k+1}(v)=w_{k}(v)$ otherwise.

At some stage  $t$ we must have  $w_t\equiv 0$. Let then  $f=\sum_{i=1}^t f_i $. Clearly, $f$ is $w$- dominating.

If  $f(u)\neq 0$ for some interval $u$,  then there exists a vertex $v$ such that $u$ is the interval extending furthest to the right and $\sum_{x\in N(v)}f(x)=w_i(v)$ for some $i$. This implies that  $u$ is maximal with respect to containment. Thus $f(v)=0$ for all  non-maximal intervals.
\begin{assertion}\label{min}
Any $w$-dominating function $h$ satisfies $\sum_{i=1}^k f(v_i)\leq \sum_{i=1}^k h(v_i)$ for all $k=1,2,\dots,n$.
\end{assertion}
\begin{proof}
By induction on $k$. For $k=1$, if $f(v_1)=0$ then clearly $f(v_1)\leq h(v_1)$. Otherwise, $f(v_1)>0$ and the definition of $f$ implies that $v_1$ is an isolated interval, namely $N(v_1)=\{v_1\}$. Since $f,h$ are $w$-dominating this implies  $f(v_1)=h(v_1)=w(v_1)$.

Assume that $\sum_{i=1}^j f(v_i)\leq \sum_{i=1}^j h(v_i)$ for every $j<k$.
If $f(v_k)=0\leq h(v_k)$ then  by induction hypothesis we have $$\sum_{i=1}^k f(v_i)=\sum_{i=1}^{k-1}f(v_i)+f(v_k)\leq \sum_{i=1}^{k-1}h(v_i)+h(v_k)=\sum_{i=1}^k h(v_i)$$ as desired.

So we may assume $f(v_k)>0$. Let $t$ be the maximal index for which $f_t(v_k)>0$. By the definition of $f_t$ there exists a vertex $v_t$ such that $v_k$ is the vertex extending furthest to the right and $f(N[v_t])=w(v_t)$. Let $j+1$ be the minimal index of the intervals in $N(v_t)$ and let $A=\{v_{j+1},\dots ,v_k\}$.
We claim that:

\begin{equation}\label{subassertion}
f[A]=f[N(v_t)].
\end{equation}
To show this, it clearly suffices to prove that $f(v)=0$  for every $v\in A\setminus N(v_t)$.
Indeed, if $v=[x,y]\in A\setminus N(v_t)$ then since $v\in A$  we have $y_{j+1}\leq y\leq y_k$ and  since $v\cap v_t=\emptyset$ we have $x_t<x$. On the other hand, since $v_t\cap v_k\neq \emptyset$ we have $x_k\leq x_t$. Hence $x_k<x<y<y_t$, meaning that $v$ is contained in $v_k$, and by the definition of $f$ it follows that $f(v)=0$. This proves \eqref{subassertion}.

Let us now return to the proof of the assertion.
Since $h$ is $w$-dominating this implies $$\sum_{i=j+1}^k f(v_i)=f[A]=f[N(v_t)]=w(v_t)\leq h[N(v_t)]\leq h[A]=\sum_{i=j+1}^k h(v_i).$$
By induction hypothesis,  $\sum_{i=1}^j f(v_i)\leq \sum_{i=1}^j h(v_i)$. Hence $$\sum_{i=1}^k f(v_i)=\sum_{i=1}^j f(v_i)+\sum_{i=j+1}^k f(v_i)\leq \sum_{i=1}^j h(v_i)+\sum_{i=j+1}^k h(v_i)=\sum_{i=1}^k h(v_i)$$ as desired.

\end{proof}

By enumerating the intervals in order of  their left endpoints and applying the same algorithm from right to left, we obtain another  $w$-dominating function $g$, satisfying $\sum_{i=k}^n g(v_i)\leq \sum_{i=k}^n h(v_i)$ for any $w$-dominating function $h$.

 Note that the enumerations of the intervals differ only by the order of the non-maximal intervals,  and that, as remarked above, the functions $f$ and $g$ take $0$ value on such intervals.

\begin{assertion}
There exists a dispersed set of vertices $I$ such that  $w[I]=|f|$.
\end{assertion}

\begin{proof}

We construct $I$ by an inductive process.
This will be accompanied by  partitioning    $V(H)$ into
sets $A_1, \ldots , A_p$, where $\{1, \ldots ,p\}$ is partitioned into two sets, $J$ and $K$. The conditions we shall demand are:

\renewcommand{\theenumi }{\Roman{enumi}}
\begin{enumerate}
\item  $I=\{v_{A_j}, ~j \in J\}$, where $v_{A_j} \in A_j$ for all $j \in J$ and $f[A_j]=g[A_j]=w(v_{A_j})$.
\item  $f[A_k]=g[A_k]=0$ for every $k \in K$.
\end{enumerate}
Note that once proved, (I) and (II) will imply

$$|f|=\sum_{j\in J}f[A_j]=\sum_{j\in J}w(v_{A_j})=w[I]$$ as desired.

To start the construction we note that if $g(v_1)=0$, then the inequality $f(v_1)\leq g(v_1)$ implies $f(v_1)=0$. Let then $I_1=\emptyset,\; A_1=\{v_1\},\; J_1=\emptyset$ and $K_1=\{1\}$.

Assume next that $g(v_1)>0$.  By the inductive definition of $g$ there exists an interval  $v_{A_1}$ for which $v_1$ is the interval extending furthest to the left among all intervals meeting $v_{A_1}$,  and   $g[N(v_{A_1})]=w(v_{A_1})$. Let $v_{j_1}$ be the vertex with the rightmost  right endpoint in $N(v_{A_1})$. Let $A_1=\{v_1,v_2,\dots,v_{j_1}\}$.
\begin{assertion}
The set $A_1$  satisfies $f[A_1]=w(v_{A_1})=g[A_1]$.
\end{assertion}
\begin{proof}
 The set $A_1$ contains $N(v_{A_1})$
and possibly some other non-maximal vertices (namely, intervals) with $g(v)=f(v)=0$. Hence, $w(v_{A_1})=g[N(v_{A_1})]=g[A_1]$. By Assertion \ref{min} we have  $g[A_1]\geq  f[A_1]$ and, since $f$ is dominating we have that $f[A_1]= f[N(v_{A_1})]\geq w(v_{A_1})$, proving the equality.
\end{proof}
Let $I_1=\{v_{A_1}\},\; J_1=\{1\} $ and $K_1=\emptyset$.

Assume that we have defined the dispersed set $I_{k-1}$,  the partition  of the set $\{v_1,v_2,\dots,v_{j_{k-1}}\}$
into sets $A_1, \ldots ,A_{k-1}$ and the partition of the set $\{1,\dots,k-1\}$ into two set $J_{k-1}$ and $K_{k-1}$, satisfying
\begin{enumerate}
\item  $I_{k-1}=\{v_{A_j}, ~j \in J_{k-1}\}$, where $v_{A_j} \in A_j$ for all $j \in J_{k-1}$ and $f[A_j]=g[A_j]=w(v_{A_j})$.
\item  $f[A_k]=g[A_k]=0$ for every $k \in K_{k-1}$.
\end{enumerate}

We then have  $$\sum_{i=1}^{j_{k-1}}f(v_i)=\sum_{i=1}^k f[A_i]= \sum_{j\in J_{k-1}}f[A_j]=\sum_{j\in J}w(v_{A_j})=w[I_{k-1}]$$.

and

$$\sum_{i=1}^{j_{k-1}}g(v_i)=\sum_{i=1}^k g[A_i]= \sum_{j\in J_{k-1}}g[A_j]=\sum_{j\in J}w(v_{A_j})=w[I_{k-1}]$$.

 Hence $\sum_{i=1}^{j_{k-1}}f(v_i)=\sum_{i=1}^{j_{k-1}}g(v_i)$.

If $g(v_{j_{k-1}+1})=0$, then since $$\sum_{i=1}^{j_{k-1}+1}g(v_i)\geq\sum_{i=1}^{j_{k-1}+1}f(v_i)$$ we have that $f(v_{j_{k-1}+1})=0$.

Let then $I_k=I_{k-1},\;  A_k=\{v_{j_{k-1}+1}\},\; J_k=J_{k-1}$ and $K_k=K_{k-1}\cup\{k\}$ .

Otherwise, $g(v_{j_{k-1}+1})>0$. By the inductive definition of $g$ there exists an interval $v_{A_k}$ for which $v_{j_{k-1}+1}$ is the interval extending furthest to the left among all the intervals meeting $v_{A_k}$, and  $g[N(v_{A_k})]=w(v_{A_k})$. Let $v_{j_k}$ be the vertex with the rightmost right endpoint in  $N(v_{A_k})$. Let $A_k=\{v_{j_{k-1}+1},\dots,v_{j_k}\}$.
\begin{assertion}
The set $A_k$  satisfies $f[A_k]=w(v_{A_k})=g[A_k]$.
\end{assertion}

\begin{proof}
The set $A_k$ includes the set   $N(v_{A_k})$
and maybe some other non-maximal vertices with $g(v)=f(v)=0$. Hence, $w(v_{A_k})=g[N(v_{A_k})]=g[A_k]$. By Assertion \ref{min} we have $\sum_{i=1}^k g(v_i)\geq \sum_{i=1}^{j_k} f(v_i)$ and by induction hypothesis we have $
\sum_{i=1}^{j_{k-1}}g(v_i)=\sum_{i=1}^{j_{k-1}}f(v_i)$. Hence, $g[A_k]\geq f[A_k]$. On the other hand, since $f$ is dominating, $ f[A_k]\geq w(v_{A_k})$ proving the desired equality.
\end{proof}
Let $I_k=I_{k-1}\cup\{v_{A_k}\},\; J_k=J_{k-1}\cup\{k\}$ and $K_k=K_{k-1}$.

The algorithm ends when for some $p$, we have  $v_{j_p}=v_n$, then \\$I=I_p=\{v_{A_j}|j\in J_p\}$ is the  dispersed set and the  partition  $A_1,\dots, A_p$ of $V$ satisfy
\begin{enumerate}
\item $f[A_j]=w(v_{A_j})$ for every $j\in J_p$, and
\item $f[A_k]=0$ for every $k\in K$.
\end{enumerate}
Hence
$$|f|=\sum_{i=1}^n f(v_i)=\sum_{i=1}^p f[A_i]=\sum_{j\in J_p} f[A_j]=\sum_{j\in J_p}w(v_{A_j})=w[I]$$ as desired.

\end{proof}
Since $I$  is  dispersed  and $f$  is $w$-dominating, we have $\gamma_w\leq|f|=w[I]\leq\rho_w$.
\end{proof}

\section{Single edge subtrees of a tree.}

In this section we prove the equality $\gamma^i_w=\gamma_w$ for  chordal graphs having a
subtrees representation  (see Theorem \ref{gavril})
in which each  subtree consists of a single edge.
As in the  case of interval graphs,  the matrix defining the two dual linear programs is not necessarily totally unimodular, and the polytopes do not necessarily have integral vertices.
For example, if the tree is a $3$ rays star, with rays of length $2$, and the edges representing the graph are all edges of the tree, then the $(0,1)$-matrix of the linear program (P) has determinant  $2$.
But here, again, the stronger $\gamma_w=\rho_w$ is true.

\begin{theorem}
The line graph $G$ of a subset of the edges of a tree  satisfies  $\gamma_w(G)=\rho_w(G)$ for every integral weights system $w$.
\end{theorem}

\begin{proof}

Let $T$ be a tree and let $F$ be a subset of its edge set. We may clearly assume that $F=E(T)$. Let $G$ be
the line graph of $F$, and let $w$ be
any system of integral weights on  $F$ (namely on the vertices of $G$).
We shall construct  a $w$-dominating function $g:V\to\mathbb{N}$ and a  dispersed set of edges $I$ with $w[I]=|g|$.

Choose a root $r$ for $T$, and direct the edges away from it. For  $v \in V(T)$ let $A(v)$ be the set of edges of the form $(v,x)$. For $e=(u,v)\in F$ let
the height, $height(e)$, be
the length (number of edges) of the longest  path from $v$ to a
leaf (so, if $v$ is a leaf, $height(e)=0$).  The depth, $depth(e)$, is the length of path from $r$ to $u$ (so, $depth(e)=0$ if and only if $e \in A(r)$).

We
define  a  $w$-dominating  function $f$ on $F$  by
induction on the height.
Let
 $f(e)=0$ for all edges $e$ of height $0$.
Assume that  $f$ has been defined on all  edges of height smaller than $k$.   For $e=(u,v)$ with  $height(e)=k$, let
$$f(e)=\max_{e'=(v,x')\in A(v)}\left( w(e')-f[A(v)]-f[A(x')]\right)^+.$$

(So, $e$ takes care of dominating its sons.) We
continue until $f$ is defined on all of $F$. By its
definition, $f$ $w$-dominates $F\setminus A(r)$. If $f$ $w$-dominates the entire $F$, then let $g=f$. Otherwise, let
$$d=\max w(e)-f[A(r)]-f[A(x)]$$

where the maximum is taken over all $e=(r,x)\in A(r)$, and let $e_0=(r,x)$  be an edge attaining this maximum.

Define $g(e_0)=f(e_0)+d$ and $g(e)=f(e)$ for
every $e\neq e_0$. Clearly, $g$ is $w$-dominating.

Next we construct $I$, the desired dispersed
set  with weight $|g|$. We
do so inductively, using the definition of $g$ above. For the first step of the induction, we distinguish two cases:

I.   $d>0$.~~
Let then $I_1=\{e_0\}$ (where $e_0$ is as defined above).

II. $d=0$. Then, for every edge $e=(r,u)\in A(r)$ satisfying
$g(e)\neq 0$ let $Z(e)$ be the set of edges
$e'=(u,v)$   satisfying  $w(e')=g[N(e')]$.
By the definition of $g$, the set $Z(e)$ is non empty. Choose an edge $e' \in Z(e)$ of maximal weight. Let $I_1$ be the set of edges $e'$ thus chosen. By its definition,   $I_1$ is dispersed.

In both cases, delete all edges in $I_1$ together with all edges intersecting them. Delete also all edges $e \in A(r)$ with
$g(e)=0$. Let $E_1$  be the set of all the deleted edges.

 Note that since  every edge $e'\in I_1$ satisfies $w(e')=g[N(e')]$,  we have  $g[E_1]=w[I_1]$.

We now apply the same procedure as above to each tree $S$ in the resulting forest  $F_1$.  Let $v_S$ be the root of $S$.  For every edge $e$  containing  $v_S$ and satisfying $g(e)\neq 0$ there exists  an edge $e'$ of depth $1$, that intersects $e$ and such that $w(e')=g[N(e')]$.
Among the edges satisfying this condition, choose the edge with maximal $g$-value. Let  $I_2$ be the set of all  edges chosen
in all trees of  $F_1$. Since $I_2$ contains only edges of depth $1$ of the new trees the set $I_1\cup I_2$  is dispersed.

We next delete all edges of  $I_2$, all their neighbors, as well as all the edges $e\in A(v_S)$ for every $S\in F_1$ that satisfy $g(e)=0$. Let $E_2$  be the set of all the deleted edges in $F_1$.

Since $w(e')=g[N(e')]$ for every $e'\in I_2$, we have $g[E_2]=w[I_2]$.

Continuing this way until all the edges are deleted, we obtain sets $I_1, I_2\ldots ,I_m$ of edges and a partition of the set of edges $E(T)$ into set $E_1,E_2,\dots,E_m$ such that $g[E_k]=w[I_k]$  for every $k=1,2,\dots,m$. Hence,  the set $I=I_1\cup I_2, \ldots, \cup I_{m}$ is dispersed, and  satisfies:
$$|g|=\sum_{k=1}^m g[E_k]=\sum_{k=1}^m w[I_k]=w[I]$$
as desired.

\end{proof}

\section{Split graphs}
Finally, we observe that
the equality $\gamma_w=\gamma^i_w$ holds in  another well-known  class of chordal graphs - split graphs.   Example \ref{ex-split} shows that the stronger equality, $\rho_w=\gamma_w$, is not necessarily true in such graphs.

\begin{definition}\label{def:split}
A graph $G=(V,E)$ is a {\em split graph} if its vertex set is the
disjoint union of  a clique and an independent set.
\end{definition}

\begin{theorem}
In a  split graph  $\gamma_w=\gamma^i_w$ for any integral weight function
$w$.
\end{theorem}

\begin{proof}
Let  $V(G)=A\cup B$ where  $A$ is a
clique and $B$ is an independent set and let $w:V\to\mathbb{N}$ be
a weight function. We have to prove that $\gamma^i_w\geq \gamma_w$.

Let $g:V\to\mathbb{N}$ be a $w$-minimal function dominating  $B$.  Assuming (as we clearly may) that there are no isolated vertices, we can assume  that $g(b)=0$ for every $b\in B$. Since $B$ is independent,  $\gamma^i_w\geq |g|$. If $|g|\geq \max_{a\in A}w(a) $, then $g$ is dominating. Hence $|g|\geq\gamma_w$, proving the desired inequality.

Suppose next that $|g| < \max_{a\in A}w(a) $. Let $a\in A$ be a vertex satisfying   $w(a) =\max\{w(a'):a'\in A\}$. Clearly, $w(a)>|g|$. Define $f$ as follows: $f(v)=g(v)$ for every $v\neq a$ and $f(a)=g(a)+w(a)-|g|$. The function $f$ is dominating and satisfies $|f|=w(a)$, so $\gamma_w\leq w(a)$.
Since clearly $\gamma^i_w \ge w(a)$, we have the desired inequality.
\end{proof}

\end{document}